\documentclass[11pt]{article}
\usepackage{amssymb}
\usepackage{amsmath}
\usepackage{dmvnbase}
\usepackage[dvips]{graphics}
\usepackage[left=3cm, right=3cm, top=2cm, bottom=2cm]{geometry}
\newtheorem{theor}{Theorem}
\newtheorem{lem}{Lemma}[section]
\newtheorem{prop}[lem]{Proposition}
\newtheorem{cor}[lem]{Corollary}

\numberwithin{equation}{section}

\def\bea{\begin{eqnarray}}
\def\eea{\end{eqnarray}}
\def\beq{\begin{equation}}
\def\eeq{\end{equation}}
\def\be{\beq\begin{array}{c}}
\def\ee{\end{array}\eeq}
\def\bse{\begin{subequations}}
\def\ese{\end{subequations}}

\def\pbea{\begin{eqnarray*}}
\def\peea{\end{eqnarray*}}
\def\pbeq{\begin{equation*}}
\def\peeq{\end{equation*}}
\def\pbe{\pbeq\begin{array}{c}}
\def\pee{\end{array}\peeq}
\def\pbse{\begin{subequations*}}
\def\pese{\end{subequations*}}
\def\a{\alpha}
\def\b{\beta}
\def\d{\delta}
\def\CC{{\mathbb C}}
\def\e{\varepsilon}
\def\D{\Delta}
\def\phi{\varphi}
\def\ov{\overline}
\def\Uqg{U_q(\widehat{\ggt})}
\def\advadva{U_q(A_2^{(2)})}
\def\advadvaC{\ov{U}_q(A_2^{(2)})}
\def\advadvaCC{\ov{\ov{U}}_q(A_2^{(2)})}
\def\rf#1{~\!\!(\ref{#1})}
\def\rfs#1{~\!\!\ref{#1}}
\def\resl{\res\limits}
\def\KK{{\mathcal K}}
\def\RR{{\mathcal R}}
\def\RRR{\bar{\RR}}

\def\Xvect{\rX}
\def\Yvect{\rY}
\def\Zvect{\rZ}

\def\Vvect{\rV}
\def\Mmat{\rM}
\def\T{\rT}

\DeclareMathOperator*{\Sym}{\rm Sym}

\newenvironment{proof}{{\flushleft\it Proof:}} {\hfill$\square$ \\}
\newenvironment{rem}{{\flushleft\it Remark.}}{}

\begin{document}
\bigskip
\hfill{ITEP-TH-64/09}
\begin{center}
{\Large\bf Weight function for the quantum \\ affine algebra $\advadva$}
\end{center}
\bigskip
\bigskip
\begin{center}
{\bf
S. Khoroshkin$^{\star\diamond}$\footnote{E-mail: khor@itep.ru}, \,
A. Shapiro$^{\star\bullet}$\footnote{E-mail: alexander.m.shapiro@gmail.com}}\\
\bigskip
{\it
  $^\star$ Institute of Theoretical \& Experimental Physics, 117259, Moscow, Russia \\
  $^\diamond$ Higher School of Economics, Myasnitskaya ul., 20, 101000, Moscow, Russia \\
  $^\bullet$ Moscow State University, Department of Mechanics and Mathematics, Leninskie gory, 119992, Moscow, Russia
}
\bigskip
\bigskip

\end{center}

\begin{abstract}
In this article, we give an explicit formula for the universal weight function
of the quantum twisted affine algebra $\advadva$. The calculations use the technique
of projecting products of Drinfeld currents onto the intersection of Borel subalgebras
of different types.
\end{abstract}

\section{Introduction}

A universal weight function of a quantum affine algebra is a family of functions with values in its Borel subalgebra satisfying certain coalgebraic properties. It can be used either to construct solutions of $q$-difference Knizhnik-Zamolodchikov equations \cite{TV} or to construct the off-shell Bethe vectors, thus generalizing the nested Bethe ansatz procedure \cite{KR}. A general construction of a weight function has been suggested in \cite{EKP}. It uses the existence of Borel subalgebras of two different types in a quantum affine algebra. One type is related to the realization of $\Uqg$ as a quantized Kac-Moody algebra, and the other comes from the ``current'' realization of $\Uqg$ proposed by Drinfeld in \cite{D}. It was proved in \cite{EKP} that the weight function of the quantum affine algebra can be represented as the projection of a product of so called Drinfeld currents on the intersection of Borel subalgebras of $\Uqg$ of different types.

In the paper \cite{T} the method of the algebraic Bethe ansatz was developed for the case of the quantum twisted affine algebra $\advadva$. In particular  there was suggested an inductive procedure for the construction of off-shell Bethe vectors for finite dimensional representations of $\advadva$.

The main result of the present paper consists in the applying of the approach of \cite{EKP} to obtain an explicit formula for the universal weight function of the quantum twisted affine algebra $\advadva$. As a corollary we derive an integral presentation for the factors of the universal $R$-martix of $\advadva$, connecting the usual and the Drinfeld comultiplications, as well as for the universal $R$-matrix itself. An analogous formula for the universal $R$-martix of $U_q(\widehat\slg_2)$ is presented in \cite{KP}.

The paper is organized as follows. In Section $2$, we describe two realizations of the quantum twisted affine algebra $\advadva$. In Section $3$, we adapt the general result of \cite{DKP1} and \cite{EKP} to the quantum twisted affine algebra $\advadva$. In Section $4$, we introduce the so called ``composite currents'' and derive their analytical properties, which are crucial for our investigations. Section $5$ contains an exposition of the main result. Here combinatorial formulae for the universal weight function and other related objects are presented. Most of the proofs are collected in Section $6$. Finally, Section $7$ contains examples of the universal weight function.

\section{Two descriptions of the quantum affine algebra $\advadva$}

\subsection{The Drinfeld realization}  \label{D-realiz}

Quantum affine algebra $\advadva$\footnote
{with zero central charge and dropped grading element.}
is an associative algebra generated by elements
$$
e_n, f_n,\;\, n\in\Z,\quad a_n,\;\, n\in\Z\setminus\hc{0},\quad {\rm and} \quad K_0^{\pm 1},
$$
subject to certain commutation relations. The relations are given as formal power series identities for the following generating functions (currents):
$$
e(z)=\sum_{k\in\Z}e_kz^{-k},\quad f(z)=\sum_{k\in\Z}f_kz^{-k},
\quad K^\pm(z)=K_0^{\pm 1}\exp\left(\pm(q-q^{-1})\sum_{n>0}a_{\pm n}z^{\mp n}\right)
$$
as follows:
\bea
(z-q^2w)(qz+w)e(z)e(w) &=& (q^2z-w)(z+qw)e(w)e(z),                           \label{ee}     \\
(q^2z-w)(z+qw)f(z)f(w) &=& (z-q^2w)(qz+w)f(w)f(z),                           \label{ff}     \\
K^+(z)e(w)K^+(z)^{-1} &=& \a(w/z)e(w),                                       \label{Kpe}    \\
K^+(z)f(w)K^+(z)^{-1} &=& \a(w/z)^{-1}f(w),                                  \label{Kpf}    \\
K^-(z)e(w)K^-(z)^{-1} &=& \a(z/w)^{-1}e(w),                                  \label{Kme}    \\
K^-(z)f(w)K^-(z)^{-1} &=& \a(z/w)f(w),                                       \label{Kmf}    \\
K^\pm(z)K^\pm(w) &=& K^\pm(w)K^\pm(z),                                       \label{KK}     \\
K^-(z)K^+(w) &=& K^+(w)K^-(z),                                               \label{KmKp}   \\
e(z)f(w)-f(w)e(z) &=& \frac1{q-q^{-1}}\hr{\d(z/w)K^+(w)-\d(z/w)K^-(z)}.      \label{ef}
\eea
where
\beq \label{a}
\a(x)=\dfrac{(q^2-x)(q^{-1}+x)}{(1-q^2x)(1+q^{-1}x)},
\eeq
and $\d\left(\dfrac zw\right)$ is a formal Laurent series, given by
\beq \label{delta}
\d(z/w) = \sum\limits_{n\in\Z}(z/w)^n.
\eeq

Generating functions $e(z),\,f(z)$ satisfy the cubic Serre relations (see \cite{D}):
\bea \label{Serre}
\Sym_{z_1,z_2,z_3} \hr{q^{-3}z_1 - (q^{-2}+q^{-1})z_2 + z_3} e(z_1)e(z_2)e(z_3) &=& 0, \\
\Sym_{z_1,z_2,z_3} \hr{q^{-3}z_1^{-1} - (q^{-2}+q^{-1})z_2^{-1} + z_3^{-1}} f(z_1)f(z_2)f(z_3) &=& 0, \\
\Sym_{z_1,z_2,z_3} \hr{q^3 z_1^{-1} - (q^{2}+q^{})z_2^{-1} + z_3^{-1}} e(z_1)e(z_2)e(z_3) &=& 0, \\
\Sym_{z_1,z_2,z_3} \hr{q^3 z_1 - (q^{2}+q^{})z_2 + z_3} f(z_1)f(z_2)f(z_3) &=& 0.
\eea
Define the {\it principal} grading on the algebra $\advadva$ by the relations
\beq \label{grading}
\deg e_n= 3n+1,\qquad \deg f_n=3n-1,\qquad \deg a_n=3n,\qquad \deg K_0^{\pm 1}=0.
\eeq
The assignment \rf{grading} defines $\advadva$ as a graded associative algebra.

The Hopf algebra structure on $\advadva$ can be defined as follows:
\bea
\D^{(D)}(e(z)) &=& e(z)\otimes 1+K^-(z)\otimes e(z), \label{DDe}\\
\D^{(D)}(f(z)) &=& 1\otimes f(z)+f(z)\otimes K^+(z), \label{DDf}\\
\D^{(D)}(K^\pm(z)) &=& K^\pm(z)\otimes K^\pm(z),     \label{DDK}\\
a^{(D)}(e(z)) &=& -(K^-(z))^{-1}e(z),                \label{aDe}\\
a^{(D)}(f(z)) &=& -f(z)(K^+(z))^{-1},                \label{aDf}\\
a^{(D)}(K^\pm(z)) &=& (K^\pm(z))^{-1},               \label{aDK}\\
\e^{(D)}(e(z)) &=& 0,                                \label{cDe}\\
\e^{(D)}(f(z)) &=& 0,                                \label{cDf}\\
\e^{(D)}(K^\pm(z)) &=& 1.                            \label{cDK}
\eea
where $\D^{(D)}$, $\e^{(D)}$ and $a^{(D)}$ are the comultiplication, the counit and the antipode maps respectively. We will call $\D^{(D)}$ a \emph{Drinfeld comultiplication}. The map $\D^{(D)}$ defines the structure of topological bialgebra on ${\mathcal A}=\advadva$ which means that for any $n\in \Z$ this map sends an element $x$ of the graded component ${\mathcal A}[n]$ to the sum of spaces $\sum_{k\geq 0} {\mathcal A}[m(n)-k]\otimes {\mathcal A}[-m(n)+k]$, where $m(n)$ is an integer such that for each $k$ the corresponding summand on the right hand side of $\D^{(D)}(x)$ is finite. See \cite[Section 2]{EKP} for details.

\subsection{Chevalley generators of $\advadva$}

Another realization of $\advadva$ can be obtained using Chevalley generators. Quantum affine algebra $\advadva$ is an associative algebra generated by elements $e_{\pm\a}$, $e_{\pm(\d-2\a)}$, $k^{\pm1}_{\a}$, $k^{\pm1}_{\d-2\a}$, satisfying the following relations:
\pbea
k_{\a}e_{\pm\a}k^{-1}_{\a} = q^{\pm1}e_{\pm\a}, &\quad&
k_{\a}e_{\pm(\d-2\a)}k^{-1}_{\a} = q^{\mp 2}e_{\pm(\d-2\a)},
\\
k_{\d-2\a}e_{\pm{\a}}k^{-1}_{\d-2\a} = q^{\mp 2}e_{\pm{\a}}, &\quad&
k_{\d-2\a}e_{\pm{(\d-2\a)}}k^{-1}_{\d-2\a} = q^{\pm4}e_{\pm{(\d-2\a)}},
\\
k_{\a}^2 k_{\d-2\a} = 1, &\quad&
\hs{e_{\pm\a},e_{\mp(\d-2\a)}} = 0,
\\
\hs{e_\a,e_{-\a}} = \dfrac{k_\a-k_\a^{-1}}{q-q^{-1}}, &\quad&
\hs{e_{\d-2\a},e_{-(\d-2\a)}} = \dfrac{k_{\d-2\a}-k_{\d-2\a}^{-1}}{q-q^{-1}},
\\
(\ad_{q}e_{\pm\a})^5 e_{\pm(\d-2\a)} = 0, &\quad&
(\ad_{q}e_{\pm(\d-2\a)})^2 e_{\pm\a} = 0,
\peea
where
\pbea
(\ad_{q}e_{\pm\a})(x) &=& e_{\pm\a}x - k_\a^{\pm 1}xk_\a^{\mp 1}e_{\pm\a},
\\
(\ad_{q}e_{\pm(\d-2\a)})(x) &=& e_{\pm(\d-2\a)}x - k_{\d-2\a}^{\pm 1}xk_{\d-2\a}^{\mp 1}e_{\pm(\d-2\a)}.
\peea
The principal grading is given by relations
\pbe
\deg e_{\pm\a}=\pm 1,\qquad \deg e_{\pm(\delta-2\a)}=\pm 1, \qquad
\deg k_\a^{\pm 1}=\deg k_{\delta-2\a}^{\pm 1}=0.
\pee

The Hopf algebra structure associated with this realization can be defined by
\pbea
\D(e_{\a})=e_{\a}\otimes 1+k_{\a}\otimes e_{\a},
&\quad&
\D(e_{\d-2\a})=e_{\d-2\a}\otimes 1+k_{\d-2\a}\otimes e_{\d-2\a},
\\
\D(e_{-\a})=1\otimes e_{-\a}+e_{-\a}\otimes k^{-1}_{\a},
&\quad&
\D(e_{-(\d-2\a)})=1\otimes e_{-(\d-2\a)}+e_{-(\d-2\a)}\otimes k^{-1}_{\d-2\a},
\\
\D(k^{\pm1}_{\a})\,=\,k^{\pm1}_{\a}\otimes k^{\pm1}_{\a},
&\quad&
\D(k^{\pm1}_{\d-2\a})\,=\,k^{\pm1}_{\d-2\a}\otimes k^{\pm1}_{\d-2\a},
\\
\e(e_{\pm\a})=0,
&\quad&
\e(e_{\pm(\d-2\a)})=0,
\\
\e(k^{\pm1}_{\a})=1,
&\quad&
\e(k^{\pm1}_{\d-2\a})=1,
\\
a(e_{\a})=-k^{-1}_{\a}e_{\a},
&\quad&
a(e_{\d-2\a})=-k^{-1}_{\d-2\a}e_{\d-2\a},
\\
a(e_{-\a})=-e_{-\a}k_{\a},
&\quad&
a(e_{-(\d-2\a)})=-e_{-(\d-2\a)}k_{\d-2\a},
\\
a(k^{\pm1}_{\a}) =  k^{\mp1}_{\a},
&\quad&
a(k^{\pm1}_{\d-2\a}) =  k^{\mp1}_{\d-2\a},
\peea
where $\D$, $\e$ and $a$ are the comultiplication, the counit and the antipode maps respectively. We will call $\D$ "standard" comultiplication.

\subsection{The isomorphism between two realizations}
The isomorphism between two descriptions of the quantum affine algebra $\advadva$ can be established by the following maps:
\beq \label{isom}
\begin{array}{rclrcl}
k_{\d-2\a}    &\to& K_0^{-2},                            &\qquad k_{\a} &\to& K_0,     \\
e_{\d-2\a}    &\to& a(qf_1f_0-f_0f_1)K_0^{-2},           &\qquad e_{\a} &\to& e_{0},   \\
e_{-(\d-2\a)} &\to& aK_0^2(q^{-1}e_0e_{-1}-e_{-1}e_0),   &\qquad e_{-\a} &\to& f_{0},
\end{array}
\eeq
where $a=\frac{1}{\sqrt{q+q^{-1}}}$.
\begin{rem} For generic $q$, the mapping above is an isomorphism of associative algebras. It preserves the counit, but does not respect the comultiplication maps $\D$ and $\D^{(D)}$. The relation between the two comultiplications is described in Section \rfs{coprod}.
\end{rem}

\section{Orthogonal decomposition and twists of Drinfeld double} \label{backgr}

In this section, we adapt the theory of orthogonal decompositions and twists of Drinfeld double developed in \cite{EKP} to the $\advadva$ case.

\subsection{Borel subalgebras in $\advadva$}\label{bor-sub}

Recall that in any quantum affine algebra there exist Borel subalgebras of two types. Borel subalgebras of the first type come from the Drinfeld (``current'') realization of $\advadva$. Let $U_F$ denote the subalgebra of $\advadva$, generated by the elements $K_0^{\pm 1},\,f_n,n\in\Z;\,a_n,n>0$, and let $U_E$ denote the subalgebra of $\advadva$, generated by the elements $K_0^{\pm 1},\,e_n,n\in\Z;\,a_n,n<0$. They are Hopf subalgebras of $\advadva$ with respect to comultiplication $\D^{(D)}$. The ``current'' Borel subalgebra $U_F$ contains the subalgebra $U_f$ generated by elements $f_n,\,n\in\Z$. The ``current'' Borel subalgebra $U_E$ contains the subalgebra $U_e$, generated by elements $e_n,\,n\in\Z$.

Borel subalgebras of the second type are obtained via the Chevalley realization. Let $U_q(\bgt_+)$ and $U_q(\bgt_-)$ denote a pair of subalgebras of $\advadva$ generated by elements
\pbe
e_{\a},\;\,e_{\d-2\a},\;\,k^{\pm1}_{\a} \qquad\text{and}\qquad
e_{-\a},\;\,e_{-(\d-2\a)},\;\,k^{\pm1}_{\a}
\pee
respectively. In terms of the Drinfeld realization, these subalgebras are generated by elements
\pbe
K_0^{\pm 1},\;\,e_0,\;\,qf_1f_0-f_0f_1 \qquad\text{and}\qquad K_0^{\pm 1},\;\,f_0,\;\,q^{-1}e_0e_{-1}-e_{-1}e_0,
\pee
respectively. The algebras $U_q(\bgt_\pm)$ are Hopf subalgebras of $\advadva$ with respect
to comultiplication~$\D$ and, moreover, coideals with respect to the Drinfeld comultiplication. More precisely:
\begin{prop} \label{prop3.1}
For any element $x\in U_q(\bgt_+)$ we have an equality
\be\label{coid}
\D^{(D)}(x)=x\otimes1+\suml{i=0}{\infty}a_i\otimes b_i \quad\;\;\text{for some}\quad
b_i\in U_q(\bgt_+),\quad\text{such that}\quad \e(b_i)=0.
\ee
\end{prop}
In particular, subalgebra $U_q(\bgt_+)$ is a left coideal of $\advadva$
with respect to coproduct $\D^{(D)}$, i.e.
$\D^{(D)}(U_q(\bgt_+))\subset \advadva\otimes U_q(\bgt_+)$.
Analogously, $U_q(\bgt_-)$ is a right coideal of $\advadva$
with respect to the same coproduct.

The proof of Proposition \rfs{prop3.1} is given in Section \rfs{section6.0}.

Let $U_F^+$, $U_f^-$, $U_e^+$ and $U_E^-$ denote the following intersections
of Borel subalgebras:

\pbea
U_f^-=U_F\cap U_q(\bgt_-), &\qquad U_F^+=U_F\cap U_q(\bgt_+), \\
U_e^+=U_E\cap U_q(\bgt_+), &\qquad U_E^-=U_E\cap U_q(\bgt_-).
\peea
Here the upper sign indicates the Borel subalgebra
$U_q(\bgt_\pm)$ containing the given algebra, and the lower letter indicates the ``current''
Borel subalgebra $U_F$ or $U_E$ that it is intersected with. These letters are capitals
if the subalgebra contains imaginary root generators $a_n$ and are lower case otherwise.
The intersections have coideal properties with respect to both comultiplications.
In particular, Proposition \rfs{prop3.1} and its analog for Borel subalgebra $U_q(\bgt_-)$ imply the inclusions
\pbea
\D^{(D)}(U_F^+) \subs U_F\otimes U_F^+, &\qquad
\D^{(D)}(U_f^-) \subs U_f^-\otimes U_F, \\
\D^{(D)}(U_E^-) \subs U_E^-\otimes U_E, &\qquad
\D^{(D)}(U_e^+) \subs U_E\otimes U^+_e.
\peea

\subsection{Projections}

Let $\Ac$ be a bialgebra with multiplication map $\mu$, comultiplication map $\delta$, unit $1$ and counit $\e$. We say that its subalgebras $\Ac_1$ and $\Ac_2$ determine an orthogonal decomposition of $\Ac$ (see \cite{ER}), if:

\begin{itemize}
\item[(i)] The algebra $\Ac$ admits a decomposition $\Ac=\Ac_1\Ac_2$, such that the multiplication map
$\mu\colon\Ac_1\otimes \Ac_2 \to \Ac$ establishes an isomorphism of linear spaces.
\item[(ii)] $\Ac_1$ is the left coideal of $\Ac$, $\Ac_2$ is the right coideal of $\Ac$:
\pbe
\delta(\Ac_1)\subs \Ac\otimes \Ac_1\ ,\qquad \delta(\Ac_2)\subs \Ac_2\otimes\Ac.
\pee
\end{itemize}

There is a pair of biorthogonal decompositions of ``current'' Borel subalgebras, equipped
with comultiplications $\D^{(D)}$ and its opposite, namely $U_E=U_e^+U_E^-$ and $U_F=U_f^-U_F^+$. The condition $(i)$ is a corollary of the theory of Cartan-Weyl bases (see \cite{KT}); the condition $(ii)$ follows from Proposition \rfs{prop3.1}.

Let $P^+$ and $P^-$ denote the projection operators corresponding to the second decomposition, so that for any $f_+\in U_F^+$ and for any $f_-\in U_f^-$
\be\label{Pdef}
P^+(f_-f_+)=\e(f_-)f_+, \qquad P^-(f_-f_+)=f_-\e(f_+).
\ee
Projection operators related to the first decomposition are denoted by $P^{*+}$ and $P^{*-}$, in such a way that for any $e_+\in U_e^+$ and for any $e_-\in U_E^-$
\be\label{Pdefa}
P^{*+}(e_+e_-)=e_+\e(e_-), \qquad P^{*-}(e_+e_-)=\e(e_+)e_-.
\ee

\subsection{Completions}\label{subsect_compl}

Algebra $\advadva$ admits a natural completion $\advadvaC$, which can be described as
the minimal extension of $\advadva$, that acts on every left highest weight and on every
right lowest weight $\advadva$-module with respect to the standard Borel subalgebras(see \cite{DKP2}).

In a highest weight representation of $\advadva$, any matrix coefficient of an arbitrary
product of currents $f(z_1)\dots f(z_n)$ is a Laurent polynomial in $\hc{z_1,\dots,z_n}$ over the ring of formal power series in variables $\hc{z_2/z_1,z_3/z_2,\dots,z_n/z_{n-1}}$. It converges to a rational function in the domain $|z_1|\gg|z_2|\gg\dots\gg|z_n|$. The latter observation and the commutation relation \rf{ff}, which
dictate the way of taking an analytic continuation from the domain above, allow us to regard
products of currents as operator-valued meromorphic functions with values in $\advadvaC$.

The projection operators $P^{\pm}$, defined in the previous subsection, extend to the completion
$\ov{U}_F\subset\advadvaC$. The projection operators $P^{*\pm}$  extend to another
completion $\ov{\ov{U}}_E\subset \advadvaCC$, where  $\advadvaCC$
denotes the minimal extension of $\advadva$ that acts in every left lowest weight and in every
right highest weight $\advadva$-module.

\subsection{Relations between coproducts and the universal $R$-matrix} \label{coprod}
Let $\RRR$ and $\KK$ denote the following formal series of elements of the tensor product $U_E\otimes U_F$:
\pbea \label{pairing}
\RRR &=& \exp\hr{(q-q^{-1})\oint e(z)\otimes f(z)\frac{dz}{z}}= 1+(q-q^{-1})\oint e(z)\otimes f(z)\frac{dz}{z} + {}\\
&&\hspace{2cm}{} + \frac{(q-q^{-1})^2}{2!}\oint\oint e(z_1)e(z_2)\otimes f(z_1)f(z_2)\frac{dz_1}{z_1}\frac{dz_2}{z_2}+\cdots\;, \\
\KK &=& \exp\left(\sum\limits_{n>0}\dfrac{n(q-q^{-1})^2}{(q^n-q^{-n})(q^n+(-1)^{n+1}+q^{-n})}
a_{-n} \otimes a_{n}\right).
\peea
By $\oint f(z){dz}$ we mean a formal integral that is equal to the coefficient of $z^{-1}$ in the formal series $f(z)$. We further set
\beq \label{Rfactors}
\RR_+=\left(P^{*-}\otimes P^+\right)(\RRR), \qquad
\RR_-=\left(P^{*+}\otimes P^-\right)(\RRR), \qquad
\text{and} \qquad
\RR=\RR_+^{21}q^{h\otimes h}\KK^{21}\RR_-,
\eeq
where $h=\log K_0$.
\begin{prop} \label{proposition0}
$(i)$ The tensor $\RR_-^{21}$ is a cocycle for $\D^{(D)}$ such that for any $x\in\advadva$
$$
\D(x)=\RR_-^{21}\D^{(D)}(x) \left(\RR_-^{21}\right)^{-1},
$$
$(ii)$ The tensor $\RR$ is the universal $R$-matrix for $\advadva$ with opposite
comultiplication $\D^{op}$:
$$
\D(x)=\RR\D^{op}(x)\RR^{-1}\quad\;\;\text{for any}\quad x\in\advadva.
$$
\end{prop}
\begin{proof}
For any element $X\in U_E\otimes U_F$, let $:\!X\!:$ denote the same element
presented as a
sum $\sum_i a_ib_i\otimes c_id_i$, where $a_i\in U_e^+$, $b_i\in U_E^-$,
$c_i\in U_f^-$, $d_i\in U_F^+$. In the work \cite{DKP1}, it was proved that the element
$\RRR'=\;:\!\RRR\!:\KK$ is a well defined series in elements of the tensor product
$\ov{\ov{U}}_E\otimes {\ov{U}}_F$, which represents the tensor
of the Hopf\footnote{which means that $\langle a_1a_2,b\rangle=\langle
 a_1\otimes a_2, \D^{op}(b)\rangle$ and $\langle a,b_1b_2\rangle=
 \langle \D(a), b_1\otimes b_2\rangle$ for any $a,a_1,a_2\in U_E$ and $b,b_1,b_2\in U_F$.}
 pairing $\ha{\;,\,}\colon U_E\otimes U_F^{op}\to\CC$, that is
$$\langle 1\otimes a,\RRR'\rangle=a,\qquad
\langle b\otimes 1,\RRR'\rangle=b$$
for any $a\in U_E$ and $b\in U_F$.\footnote{In \cite{DKP1} this result was derived
for $U_q(\widehat{{\mathfrak{sl}}}_2)$. However, the arguments therein are general and work
for any quantum affine algebra.}
Due to Proposition \rfs{proposition0} and to the theory of Cartan-Weyl bases
for $\advadva$ developed in \cite{KT},\footnote{To adapt the result of \cite{KT} to our case, one should replace $q$-commutators with $q^{-1}$-commutators in the construction of the Cartan-Weyl basis.}
the pair $(U_e^+,U_E^-)$ forms an orthogonal decomposition of the bialgebra
$U_E$. Analogously, the dual pair $(U_f^-,U_F^+)$ forms an orthogonal decomposition of the bialgebra $U_F$. This implies the decomposition $\RRR'=\RRR'_1\RRR'_2$, where $\RRR'_1$ is the tensor of the pairing of $U_e^+$ and $U_f^-$,
$\RRR'_2$ is the tensor of the pairing of $U_E^-$ and $U_F^+$.
They are given by the relations:
\begin{equation}
 \begin{split}
 \label{Rplus}
 \RRR'_1 &= 1\otimes P^- (\RRR')=P^{*+}\otimes P^- (\RRR') =
 P^{*+}\otimes P^- (\RRR),
 \\
 \RRR'_2 &= 1\otimes P^+ (\RRR')=P^{*-}\otimes P^+ (\RRR')=
 \left(P^{*-}\otimes P^+ (\RRR)\right)\KK.
 \end{split}
\end{equation}
The last equalities in both lines of \rf{Rplus} are implied by the following observation:
the application of projection operators to both tensor components automatically performs the
 normal ordering of the tensor $\RRR$. The factorization of the tensor $\KK$ follows
 from the definitions of the projections.

Now, part (i) of the proposition becomes a direct generalization of Proposition 3.8 of \cite{EKP}.
 One can check that the conditions (H1)--(H6) of Section 2.3 of \cite{EKP}
 are satisfied for the algebras $A_1=U_e^+$, $A_2=U_E^-$,
$B_1=U_F^+$, $B_2=U_f^-$ due to the theory of Cartan-Weyl bases, Proposition~\ref{proposition0}
and the relations \rf{Rplus}. Under the same settings part (ii) is the direct generalization
of the results of \cite[Section 2.3]{EKP}.
\end{proof}

\begin{rem}
Note that all the results of this and of previous sections could be easily modified for the central extended  algebra $\advadva$ with an added grading element. See e.g. \cite{DK}.
\end{rem}

\subsection{The universal weight function}

Let $V$ be a representation of the algebra $\advadva$ and $v$ be a vector in $V$. We call $v$
a highest weight vector with respect to the current Borel subalgebra $U_E$ if
\pbea
e(z)v &=& 0, \\
K^\pm(z)v &=& \la(z)v,
\peea
where $\la(z)$ is a meromorphic function, decomposed in series in $z^{-1}$ for $K^+(z)$
and in series in $z$ for $K^-(z)$. Representation $V$ is called a representation with
highest weight vector $v \in V$ with respect to $U_E$ if it is generated by $v$ over $\advadva$.
Suppose that for every finite ordered set $I=\hc{i_1,\dots,i_n}$, an element
$W(z_{i_1},\dots,z_{i_n})$ is chosen. It is a Laurent polynomial in $\hc{z_{i_1},\dots,z_{i_n}}$ over the ring $\advadva\hs{\hs{z_{i_2}/z_{i_1}, z_{i_3}/z_{i_2},\dots,z_{i_n}/z_{i_{n-1}}}}$ satisfying the following conditions:
\begin{enumerate}
\item[(1)] for any representation $V$ that is a highest weight representation with respect to $U_E$ with highest weight vector $v$ the function
\pbe
w_V(z_{i_1},\dots,z_{i_n})=W(z_{i_1},\dots,z_{i_n})v
\pee
converges in the domain $|z_{i_1}|\gg\dots\gg|z_{i_n}|$ to a meromorphic $V$-valued function;
\item[(2)] $W=1$ if $I=\es$;
\item[(3)] let $V=V_1 \otimes V_2$ be a tensor product of the highest weight representations
with highest weight vectors $v_1$ and $v_2$ and the highest weight series $\hc{\la^{(1)}(z)}$
and $\hc{\la^{(2)}(z)}$; then for an ordered set $I$ we have
\begin{multline}
w_V\hr{\hc{z_{i|i \in I}}} = \suml{I=I_1 \bigsqcup I_2}{}w_{V_1}\hr{\hc{z_{i|i \in I_1}}}
\otimes w_{V_2}\hr{\hc{z_{i|i\in I_2}}} \times \\ \times \prodl{i \in I_1}{}\la^{(2)}(z_i)
\quad\times \prodl{\substack{i < j, \\ i \in I_1, j \in I_2}}{}
\dfrac{(q^{-2}-z_j/z_i)(q+z_j/z_i)}{(1-q^{-2}z_j/z_i)(1+qz_j/z_i)}.
\end{multline}
\end{enumerate}
A collection $W(z_{i_1},\dots,z_{i_n})$ for all $n$ is called the universal weight function.
Using Proposition~\ref{proposition0}, we get the following direct generalization of \cite[Theorem 3]{EKP}:
\begin{prop}\label{proposition000}
The collection $P^+\br{f(z_{i_1}),\dots,f(z_{i_n})}$ is a weight function for $\advadva$:
\beq
W(z_{i_1},\dots,z_{i_n}) = P^+\br{f(z_{i_1}),\dots,f(z_{i_n})}.
\eeq
\end{prop}
 Our main goal from now on is to calculate the latter projection.

\section{Composite currents}

In the following sections we use the abbreviated notation $P$ for the projection $P^+$.

\subsection{Definitions} \label{cc}

We introduce a pair of new currents $s(z)$ and $t(z).$ The first one plays
a crucial role in the following discussion since the main result is formulated
in terms of the projection $P(s(z)).$ Thus $s(z)$ and $t(z)$ are a pair
of Laurent series with coefficients in the completed algebra $\ov U_F$ given by
\bea \label{s}
s(z) &=& \res\limits_{w=-q^{-1} z}f(z)f(w)\frac{dw}w, \\
t(z) &=& \res\limits_{w=q^2 z}f(z)f(w)\frac{dw}w.
\eea
We call generating functions $s(z)$, $t(z)$ the \emph{composite currents}. We treat them as meromorphic operator-valued functions in the category of left highest weight $\advadva$-modules.
\begin{prop}
In $\ov U_F$ the composite currents $s(z)$ and $t(z)$ can be represented as the following products:
\bea
\label{s}
s(z) &=& \frac{(q+q^2)(q^{-1}-q)}{1+q^3}f(-q^{-1}z)f(z), \\
\label{t}
t(z) &=& \frac{(q+q^2)(q^2-q^{-2})}{1+q^3}f(q^2z)f(z).
\eea
\end{prop}
\begin{proof}
\begin{multline*}
\res\limits_{w=-q^{-1} z}f(z)f(w)\frac{dw}w = \lim\limits_{w\to -q^{-1}z}
\hr{(w+q^{-1}z)f(z)f(w)\frac1w}
= \\ = \lim\limits_{w\to -q^{-1}z}\hr{\frac{(z-q^2w)(qz+w)}{(q^2z-w)w}f(w)f(z)} =
\frac{(q+q^2)(q^{-1}-q)}{1+q^3}f(-q^{-1}z)f(z).
\end{multline*}
The second equality is proved in a similar way.
\end{proof}

\subsection{Analytic properties}

We use the following vanishing conditions on products of Drinfeld currents: \emph{diagonal conditions}~(i) follow from relation \rf{ff} and \emph{Serre conditions}~(ii) follow from analytic Serre relations~\rf{Serre} (see \cite{DK}):
\begin{itemize}
\item[(i)] the product $f(z)f(w)$ has a simple zero on the hyperplane $z=w$;
\item[(ii)] the product $f(z_1)f(z_2)f(z_3)$ has a simple zero on the following lines:
$z_1=-qz_3=q^2z_2$, $z_2=-qz_1=q^2z_3\;$ and $\;z_3=-qz_2=q^2z_1$.
\end{itemize}
We collect data derived from the vanishing conditions in the following table:
\begin{center}
\begin{tabular}{|c|c||c|c|}
\hline
\multicolumn{2}{|c||}{$s(z)f(w)$} & \multicolumn{2}{|c|}{$f(w)s(z)$} \\
\hline
zeros & poles & zeros & poles\\
\hline
$w=z$ & $w=q^2z$ & $w=z$ & $w=z$\\
$w=-q^{-1}z$ & $w=-q^{-1}z$ & $w=-q^{-1}z$ & $w=-q^{-3}z$\\
$w=-qz$ & $w=-qz$ & $w=-qz$ & $w=-qz$\\
$w=q^{-2}z$ & $w=q^{-2}z$ & $w=q^{-2}z$ & $w=q^{-2}z$\\
\hline
\end{tabular}
\end{center}
where the first pair of zeros is a consequence of the diagonal conditions while the second pair is obtained from the Serre relations. Cancelling zeros and poles at the same points we derive that $s(z)f(w)$ has just one simple zero surviving on $w=z$ and just one simple pole on $w=q^2z$. Analogously, $f(w)s(z)$ has a simple zero on $w=-q^{-1}z$ and a simple pole on $w=-q^{-3}z$. Therefore the following equality of holomorphic functions holds:
\be \label{sf}
\dfrac{(q^2z-w)}{(z-w)}s(z)f(w)=\dfrac{(z+q^3w)}{(z+qw)}f(w)s(z).
\ee

In a similar way, we obtain one more table:
\begin{center}
\begin{tabular}{|c|c||c|c|}
\hline
\multicolumn{2}{|c||}{$s(z)s(w)$} & \multicolumn{2}{|c|}{$s(w)s(z)$} \\
\hline
zeros & poles & zeros & poles\\
\hline
$w=z$        & $w=q^2z$     & $w=z$        & $w=q^{-2}z$\\
$w=-qz$      & $w=-q^3z$    & $w=-q^{-1}z$ & $w=-q^{-3}z$\\
$w=z$        & $w=z$        & $w=z$        & $w=z$\\
$w=-q^{-1}z$ & $w=-q^{-1}z$ & $w=-qz$      & $w=-qz$\\
$w=q^2z$     & $w=q^2z$     & $w=q^{-2}z$  & $w=q^{-2}z$\\
$w=-q^{-1}z$ & $w=-q^{-1}z$ & $w=-qz$      & $w=-qz$\\
$w=-qz$      & $w=-qz$      & $w=-q^{-1}z$ & $w=-q^{-1}z$\\
$w=q^{-2}z$  & $w=q^{-2}z$  & $w=q^2z$     & $w=q^2z$\\
\hline
\end{tabular}
\end{center}
In this case the first four zeros are derived from diagonal conditions and the rest from Serre relations.
Thus the product $s(z)s(w)$ has simple zeros at the points $w=z$ and $w=-qz$ and simple poles at the points $w=q^2z$ and $w=-q^3z$. Once again we obtain the following equality of holomorphic functions:
\be \label{ss}
\dfrac{(q^2z-w)(q^3z+w)}{(z-w)(qz+w)}s(z)s(w) = \dfrac{(z-q^2w)(z+q^3w)}{(z-w)(z+qw)}s(w)s(z).
\ee

\section{Calculation of the weight function}

\subsection{Required notation}

In this section we introduce some notation that is frequently used below.
The following rational functions are formal power series, converging in the domain $|z_1|\gg\dots\gg|z_n|$.
\begin{align}
\label{rho}
\rho_k(z_1,\dots,z_{n-1};z_n) &=
\prodl{\substack{i=1 \\ i\ne k}}{n-1}\dfrac{z_n-z_i}{z_k-z_i}\prodl{i=1}{n-1}\dfrac{z_k-q^2z_i}{z_n-q^2z_i},
\\ \label{la}
\la_k(z_1,\dots,z_{n-1};z_n) &=
\dfrac{z_k}{qz_n+z_k}\prodl{i=1}{n-1}\dfrac{(z_n-z_i)(z_k+q^3z_i)}{(z_k+qz_i)(z_n-q^2z_i)},
\\ \label{mu}
\mu_k(z_1,\dots,z_{n-1};z_n) &=
\prodl{\substack{i=1 \\ i\ne k}}{n-1}\dfrac{z_n-z_i}{z_k-z_i}
\prodl{i=1}{n-1}\dfrac{(z_n+qz_i)(z_k-q^2z_i)(z_k+q^3z_i)}{(z_k+qz_i)(z_n-q^2z_i)(z_n+q^3z_i)},
\\ \label{nu}
\nu_k(z_1,\dots,z_{n-1};z_n) &=
-q^n\prodl{\substack{i=1 \\ i\ne k}}{n-1}\dfrac{z_n+qz_i}{z_k-z_i}
\prodl{i=1}{n-1}\dfrac{(z_n-z_i)(z_k+qz_i)(z_k-q^2z_i)}{(qz_k+z_i)(z_n-q^2z_i)(z_n+q^3z_i)}.
\end{align}
Let $\Fc(z_1,\dots,z_{n-1};z_n)$ and $\Sc(z_1,\dots,z_{n-1};z_n)$ denote the following combinations of projections:
\beq \label{Fc}
\Fc(z_1,\dots,z_{n-1};z_n)=P(f(z_n))-\suml{k=1}{n-1}\rho_k(z_1,\dots,z_{n-1};z_n)P(f(z_k)),
\eeq
\beq \label{Sc}
\begin{split}
\Sc(z_1,\dots,z_{n-1};z_n)=P(s(z_n)) &- \suml{k=1}{n-1}\mu_k(z_1,\dots,z_{n-1};z_n)P(s(z_k)) - {}\\
{} &- \suml{k=1}{n-1}\nu_k(z_1,\dots,z_{n-1};z_n)P(s(-qz_k)).
\end{split}
\eeq
Let $I=\hc{i_1,\dots,i_r}$ and $J=\hc{j_1,\dots,j_r}$ be ordered subsets of an ordered set $\hc{1,\dots,n}$. We call an ordered pair of subsets $\hc{I,J}$ a \emph{$P^+$-admissible pair} of cardinality $r$ if the following hold:
\begin{itemize}
\item $I \cap J = \es$;
\item $j_1>j_2>\dots>j_r$;
\item $j_\ell>i_\ell$, for $\ell=1,\dots,r$.
\end{itemize}
For each $P^+$-admissible pair $\hc{I,J}$ of cardinality $r$, and for any $k=1,\dots,r$, we introduce a formal power series $\tau_{I,J}^k\hr{z_1,\dots,z_n}$ that is the decomposition in the domain $|z_1|\gg|z_2|\gg\dots\gg|z_n|$ of the rational function
\beq
\label{tau-fin}
\tau_{I,J}^k\hr{z_1,\dots,z_n} =
-\la_{i_k}\hr{z_{i_1},\dots,z_{i_{k-1}},z_1,\dots,z_{j_{k}-1};z_{j_k}}
\prodl{\substack{\ell=1\\ \ell\ne i_1,\dots,i_{k-1}}}{i_k-1}\a\hr{\dfrac{z_\ell}{z_{i_k}}}
\prodl{\substack{\ell=1\\ \ell\ne i_1,\dots,i_k}}{j_k-1}\a\hr{\dfrac{-qz_\ell}{z_{i_k}}},
\eeq
where $z_{i_1},\dots,z_{i_{k-1}}$ are skipped in the row $z_1,\dots,z_{j_k-1}$ of variables of the function $\la_{i_k}$ above. In other words, each of $z_{i_1},\dots,z_{i_{k-1}}$ occur in the variables of function $\la_{i_k}$, but only once.

Now, for each $P^+$-admissible pair $\hc{I,J}$ of cardinality $r$ and for any $k=1,\dots,n$ such that $k\notin I\cup J$, we introduce a generating function $\Fc_{I,J}^k(z_1,\dots,z_n)$. Namely, for such a $k$ there exists a unique $p=1,\dots,r+1$ for which $j_p<k<j_{p-1}$, with $j_0=n+1$, $j_{r+1}=0$. Then we set
\beq \label{Fc-fin}
\Fc_{I,J}^k(z_1,\dots,z_n)=\Fc\hr{z_{i_1},\dots,z_{i_{p-1}},z_1,\dots,z_{k-1};z_k},
\eeq
where again $z_{i_1},\dots,z_{i_{p-1}}$ are skipped in the row $z_1,\dots,z_{k-1}$.

Finally, we define a q-commutator
\footnote{please note that the definition differs from the one given in \cite{KT}.}
\beq \label{qcomm}
\hs{a,b}_{q^{\pm1}} = ab-q^{\pm1}ba
\eeq
and a pair of ordered products
$$
\prodl{k=1}{\substack{r \\ \longra}}G_k = G_1G_2\dots G_r,
\qquad
\prodl{k=1}{\substack{r \\ \longla}}G_k = G_r\dots G_2G_1.
$$

\subsection{Main results}

\begin{theor} \label{Th1}
Projections of currents $f(z)$ and $s(z)$ can be written as follows:
\bea
P(f(z)) &=& \sum\limits_{n>0}f_nz^{-n}, \\
P(s(z)) &=& -\dfrac1{q+q^{-2}}\hr{q\hs{P(f(z)),f_0}_{q^{-1}} + \hs{f_1z^{-1},f_0+P(f(z))}_{q^{-1}}}.
\eea
\end{theor}

\begin{theor} \label{Th-main}
The projection $P(f(z_1)\cdots f(z_n))$ can be expressed by the following explicit formula:
\begin{multline} \label{ans}
P\hr{f(z_1)\cdots f(z_n)} =
\suml{r=0}{\hs{n/2}}\suml{\substack{\hc{I,J}\\ \hm{I}=\hm{J}=r}}{}
\prodl{m=1}{r}\tau_{I,J}^m\hr{z_1,\dots,z_n}
 \times {} \\ {} \times
\prodl{k=1}{\substack{r \\ \longra}}\Sc\hr{z_{i_1},\dots,z_{i_{k-1}};z_{i_k}}
\prodl{\substack{\ell=1 \\ \ell\notin I,J}}{\substack{n \\ \longra}}\Fc_{I,J}^\ell\hr{z_1,\dots,z_n},
\end{multline}
where $\Sc(z_1,\dots,z_{n-1};z_n)$, $\tau_{I,J}^k\hr{z_1,\dots,z_n}$, $\Fc_{I,J}^k\hr{z_1,\dots,z_n}$ are defined by \rf{Sc}, \rf{tau-fin}, \rf{Fc-fin} respectively. The second sum in the formula \rf{ans} is taken over $P^+$-admissible pairs $\hc{I,J}$ of cardinality $r$, and in the case $I=J=\es$ we have only one summand
$$
\prodl{\ell=1}{\substack{n \\ \longra}}\Fc\hr{z_1,\dots,z_{\ell-1};z_\ell}.
$$
\end{theor}

We illustrate Theorem \rfs{Th-main} by giving examples in Section \rfs{Examples}.

\subsection{Calculation of other projections}

In this section we obtain some formulae for projections $P^-,\,P^{*+},\,P^{*-}$. The proofs almost literally reproduce the ones of Theorem \rfs{Th-main}, and we skip them.

Let us introduce a composite current
$$
\tilde s(z) = \res\limits_{w=-qz}f(w)f(z)\frac{dw}w.
$$
As a matter of fact $\tilde s(z) = -s(-qz)$, but $\tilde s(z)$ is more convenient for calculating projection $P^-$.
The following theorem holds:
\begin{theor} Projections $P^-(f(z))$ and $P^-(\tilde s(z))$ can be written as follows:
\bea
P^-(f(z)) &=& \sums{n\le0}f_nz^{-n}, \\
P^-(\tilde s(z)) &=& \dfrac1{1+q^3}\hr{\hs{f_0,P^-(f(z))}_q + q\hs{P^-(f(z))-f_0,f_1z^{-1}}_q}.
\eea
\end{theor}
Now we introduce some more notation:
\begin{align}
\label{rho-}
\tilde \rho_k(z_1;z_2,\dots,z_n) &=
\prodl{\substack{i=2 \\ i\ne k}}{n}\dfrac{z_1-z_i}{z_k-z_i}
\prodl{i=2}{n}\dfrac{q^2z_k-z_i}{q^2z_1-z_i},
\\ \label{la-}
\tilde \la_k(z_1;z_2,\dots,z_n) &=
-\dfrac{qz_k}{z_1+qz_k}\prodl{i=2}{n}\dfrac{(z_1-z_i)(q^3z_k+z_i)}{(qz_k+z_i)(q^2z_1-z_i)},
\\ \label{mu-}
\tilde\mu_k(z_1;z_2,\dots,z_n) &=
\prodl{\substack{i=2 \\ i\ne k}}{n}\dfrac{z_1-z_i}{z_k-z_i}
\prodl{i=2}{n}\dfrac{(qz_1+z_i)(q^2z_k-z_i)(q^3z_k+z_i)}{(qz_k+z_i)(q^2z_1-z_i)(q^3z_1+z_i)},
\\ \label{nu-}
\tilde\nu_k(z_1;z_2,\dots,z_n) &=
-q^{n-1}\prodl{\substack{i=2 \\ i\ne k}}{n}\dfrac{qz_1+z_i}{z_k-z_i}
\prodl{i=2}{n}\dfrac{(z_1-z_i)(qz_k+z_i)(q^2z_k-z_i)}{(z_k+qz_i)(q^2z_1-z_i)(q^3z_1+z_i)}.
\end{align}
Define the following combinations of projections:
\beq \label{Fc-}
\tilde\Fc(z_1;z_2,\dots,z_n)= P^-(f(z_1))-\suml{k=2}{n}\tilde\rho_k(z_1;z_2,\dots,z_n)P^-(f(z_k)),
\eeq
\beq \label{Sc-}
\begin{split}
\tilde\Sc(z_1;z_2,\dots,z_n)=
P^-(\tilde s(z_1)) &- \suml{k=2}{n}\tilde\mu_k(z_1;z_2,\dots,z_n)P^-(\tilde s(z_k)) - {} \\
{} &- \suml{k=2}{n}\tilde\nu_k(z_1;z_2,\dots,z_n)P^-(\tilde s(-q^{-1}z_k)),
\end{split}
\eeq

Now, let $I=\hc{i_1,\dots,i_r}$ and $J=\hc{j_1,\dots,j_r}$ be ordered subsets of an ordered set $\hc{1,\dots,n}$. We call an ordered pair of subsets $\hc{I,J}$ a \emph{$P^-$-admissible pair} of cardinality $r$ if the following hold:
\begin{itemize}
\item $I \cap J = \es$;
\item $i_1<i_2<\dots<i_r$;
\item $i_\ell<j_\ell$, for $\ell=1,\dots,r$.
\end{itemize}
For each $P^-$-admissible pair $\hc{I,J}$ of cardinality $r$, and for any $k=1,\dots,r$, we introduce a formal power series $\tilde\tau_{I,J}^k\hr{z_1,\dots,z_n}$ that is the decomposition in the domain $|z_1|\gg|z_2|\gg\dots\gg|z_n|$ of the rational function
\beq
\label{tau-fin-}
\tilde\tau_{I,J}^k\hr{z_1,\dots,z_n} =
\tilde\la_{j_k}\hr{z_{i_k};z_{i_k+1},\dots,z_n,z_{j_{k-1}},\dots,z_{j_1}}
\prodl{\substack{\ell=j_k+1\\ \ell\ne j_1,\dots,j_{k-1}}}{n}\a\hr{\dfrac{z_{j_k}}{z_\ell}}
\prodl{\substack{\ell=i_k+1\\ \ell\ne j_1,\dots,j_k}}{n}\a\hr{\dfrac{-qz_{j_k}}{z_\ell}},
\eeq
where $z_{j_1},\dots,z_{j_{k-1}}$ are skipped in the row $z_{i_k+1},\dots,z_n$ of variables of the function $\tilde\la_{i_k}$ above. In other words each of $z_{j_1},\dots,z_{j_{k-1}}$ occur in the variables of function $\tilde\la_{i_k}$, but only once.

Now, for each $P^-$-admissible pair $\hc{I,J}$ of cardinality $r$ and for any $k=1,\dots,n$ such that $k\notin I,J$, we introduce a generating function $\tilde\Fc_{I,J}^k(z_1,\dots,z_n)$. Namely, for such a $k$ there exists a unique $p=1,\dots,r+1$ for which $i_{p-1}<k<i_p$, with $i_0=0$, $i_{r+1}=n+1$. Then we set
\beq \label{Fc-fin-}
\tilde\Fc_{I,J}^k(z_1,\dots,z_n)=\tilde\Fc\hr{z_k;z_{k+1},\dots,z_n,z_{j_{p-1}},\dots,z_{j_1}},
\eeq
where again $z_{j_1},\dots,z_{j_{p-1}}$ are skipped in the row $z_{k+1},\dots,z_n$.

Then we have the following theorem:
\begin{theor} \label{theor-}
The projection $P^-(f(z_1)\cdots f(z_n))$ can be expressed by the following explicit formula:
\begin{multline} \label{ans-}
P^-\hr{f(z_1)\cdots f(z_n)} = \suml{r=0}{\hs{n/2}}\suml{\substack{\hc{I,J}\\ \hm{I}=\hm{J}=r}}{}
\prodl{m=1}{r}\tilde\tau_{I,J}^m\hr{z_1,\dots,z_n}
 \times {} \\ {} \times
\prodl{\substack{\ell=1 \\ \ell\notin I,J}}{\substack{n \\ \longra}}\tilde\Fc_{I,J}^\ell\hr{z_1,\dots,z_n}
\prodl{k=1}{\substack{r \\ \longla}}\tilde\Sc\hr{z_{j_k};z_{j_{k-1}},\dots,z_{j_1}},
\end{multline}
where $\tilde\Sc(z_1;z_2\dots,z_n)$, $\tilde\tau_{I,J}^k\hr{z_1,\dots,z_n}$, $\tilde\Fc_{I,J}^k\hr{z_1,\dots,z_n}$ are defined by \rf{Sc-}, \rf{tau-fin-}, \rf{Fc-fin-} respectively. The second sum in the formula \rf{ans-} is taken over $P^-$-admissible pairs $\hc{I,J}$, and in the case $I=J=\es$ we have only one summand
$$
\prodl{\ell=1}{\substack{n \\ \longra}}\tilde\Fc\hr{z_\ell;z_{\ell+1},\dots,z_n}.
$$
\end{theor}

To obtain the formulae for projections $P^{*+}$ and $P^{*-}$ we introduce an involution $\iota$ of the algebra $\advadva$ such that
\beq
\iota(e_n) = f_{-n}, \qquad \iota(f_n) = e_{-n}, \qquad \iota(a_n) = a_{-n}, \qquad \iota(K_0) = K_0^{-1}.
\eeq
One can see that
\beq \label{Pstar}
\iota P^\pm = P^{*\mp}\iota.
\eeq
Using this identity we can compute $P^{*\pm}(e(z_1)\dots e(z_n))$ as $\iota P^{\mp}(f(z_1^{-1})\dots f(z_n^{-1}))$.

\begin{cor}
An integral presentation for the factors of the universal $\Rc$-matrix of $\advadva$ can be obtained by applying formulae \rf{ans}, \rf{ans-} and \rf{Pstar} to relations \rf{Rfactors}.
\end{cor}

\section{Proofs}

\subsection{Proof of Proposition \ref{proposition0}}\label{section6.0}

The statement is multiplicative with respect to $x$; hence it is sufficient to prove it for
 generators $e_\a$, $e_{\d-2\a}$, $k_\a^{\pm 1}$ of the subalgebra $U_q(\bgt_+)$.
For elements $e_\a$ and $k_\a^{\pm 1}$ the property \rf{coid} follows directly
from the isomorphism between the realizations \rf{isom} and the formulae of the comultiplication $\D^{(D)}$ \rf{DDe}--\rf{cDK}. Thus we only need to check the property for $e_{\d-2\a}$.

Let $d$ be a grading on the commutative algebra, generated
by the elements $a_i,i\in\N$ and~$K_0$, such that $d(a_i)=i$ and $d(K_0)=0$.
Using Drinfeld comultiplication formulae we obtain
\begin{align*}
\D^{(D)}(f_0) &= 1\otimes f_0 + \suml{n\ge0}{}f_{-n}\otimes K_0I_n, \\
\D^{(D)}(f_1) &= 1\otimes f_0 + \suml{n\ge0}{}f_{-n+1}\otimes K_0J_n,
\end{align*}
where $I_n, J_n \in \Cbb(q)[a_1, a_2, \dots]$ are polynomials in variables $a_i, i\in\N$
of degree $n$ over the quotient field~$\Cbb(q)$. According to \rf{isom} it only remains
to check the statement of the lemma for
$$
x=qf_1f_0-f_0f_1.
$$
The equality
\pbe
\D^{(D)}(qf_1f_0-f_0f_1) = 1\otimes(qf_1f_0-f_0f_1) + \suml{n\ge0}{}f_{-i+1}\otimes
(qJ_nf_0-f_0J_n)
\pee
holds modulo $\advadva\otimes U_q(\bgt_+)$. Formulae \rf{isom} imply that
$$
1\otimes(qf_1f_0-f_0f_1) \in \advadva\otimes U_q(\bgt_+).
$$ To prove the same for the second
summand we use equality
$$
K^+(z)f_0=q^{-1}f_0K^+(z) \mod\; U_q(\bgt_+),
$$
which one can obtain from relation \rf{Kpf}. \hfill{$\square$}

\subsection{Proof of Theorem 1}

The first part of the theorem immediately follows from the definition of the projection $P^+$. The proof of the second part is a little bit more complicated. We start with an equality
$$
\resl_{w=-q^{-1}z}\hr{f(z)f(w)\frac{dw}w}=\dfrac1{(q^2+q^{-1})z}\resl_{w=-q^{-1}z}
\hr{(q^2z-w)f(z)f(w)\frac{dw}w}.
$$
Using the Cauchy Residue Theorem we rewrite the latter residue as the following sum:
\begin{multline}\notag
\resl_{w=-q^{-1}z}\hr{(q^2z-w)f(z)f(w)\frac{dw}w} =\\= -\hr{\resl_{w=\infty}
\hr{(q^2z-w)f(z)f(w)\frac{dw}w} + \resl_{w=0}\hr{(q^2z-w)f(z)f(w)\frac{dw}w}}.
\end{multline}
It remains to calculate the residues at $w=0$ and $w=\infty$.
$$
\resl_{w=0}\hr{(q^2z-w)f(z)f(w)\dfrac{dw}w}=q^2zf(z)f_0-f(z)f_1.
$$
Since $f(z)f(w)$ is well defined only in the domain $|z|\gg|w|,$ we evaluate the residue
 of its analytic continuation at $w=\infty$\;:
\pbe
\resl_{w=\infty}\hr{(q^2z-w)f(z)f(w)\dfrac{dw}w} = \resl_{w=\infty}
\hr{\dfrac{(z/w -q^2)(z/w + q^{-1})}{(1+q^{-1}z/w)}f(w)f(z)dw},
\pee
which after putting $\;u=\dfrac1w\;$ equals
\begin{multline*}
\resl_{u=0}\hr{\dfrac{\hr{q^2-uz}\hr{q^{-1}+uz}}{\hr{1+q^{-1}uz}}f\hr{\dfrac1u}
f(z)\dfrac{du}{u^2}} = \\ =
z\hr{qf_1z^{-1}f(z)-qf_0f(z)+\hr{q^2+q-1-q^{-1}}\suml{n\ge0}{}f_{-n}(-q^{-1}z)^n f(z)}.
\end{multline*}
Previous relations imply
\begin{multline} \label{sdecomp}
s(z)=-\dfrac{1}{q^2+q^{-1}}\left(q^2\hs{f(z),f_0}_{q^{-1}}+q\hs{f_1z^{-1},f(z)}_{q^{-1}}+{}\right.\\
\left.{}+\hr{q^2+q-1-q^{-1}}P^-\hr{f(-q^{-1}z)}f(z)\right),
\end{multline}
where $\hs{a,b}_{q^{-1}}$ is defined by \rf{qcomm}.
Recalling projection properties \rf{Pdef},
$$
P(s(z)) = -\dfrac{1}{q^2+q^{-1}}P\hr{q^2\hs{f(z),f_0}_{q^{-1}}+q\hs{f_1z^{-1},f(z)}_{q^{-1}}}.
$$
To finish the proof we rewrite relation \rf{ff} in coordinates:
\be\label{ffcomp}
q^2f_{n+2}f_m-qf_mf_{n+2}+q^2f_{m+2}f_n-qf_{n}f_{m+2} = (1-q^3)(f_{m+1}f_{n+1}+f_{n+1}f_{m+1}).
\ee
Using relation \rf{ffcomp}, the projection properties \rf{Pdef} and the Borel subalgebras description,
 given in the Section \rfs{backgr}, we obtain that
\begin{align*}
qf_1f_n-f_nf_1 \in U_F^+, &\quad\text{for }n\ge0, \\
qf_mf_0-f_0f_m \in U_F^+, &\quad\text{for }m>0.
\end{align*}
For $n>0$ the relation above is obvious, for $n=0$, $m=1$ it follows from the formulae \rf{isom},
and finally for $n<0$, $m\ne1$ equality \rf{ffcomp} provides us with the desired relation.
Therefore we obtain
\pbe
P(s(z))=-\dfrac1{q+q^{-2}}\hr{q\hs{P(f(z)),f_0}_{q^{-1}} + \hs{f_1z^{-1},f_0+P(f(z))}_{q^{-1}}}.
\pee

\vspace{-5mm}
\hfill{$\square$}

\subsection{Proof of Theorem 2}
Assume $0\le p<k\le n-1$. We introduce the formal power series $\tau_{k,p}(z_1,\dots,z_{n-1};z_n)$, which is the decomposition in the domain $|z_1|\gg|z_2|\gg\dots\gg|z_n|$ of the rational function
\beq \label{tau}
\tau_{k,p}(z_1,\dots,z_{n-1};z_n) = -\la_k(z_1,\dots,z_{n-1};z_n)
\prodl{\ell=p+1}{k-1}\a\hr{\dfrac{z_\ell}{z_k}}
\prodl{\substack{\ell=p+1 \\ \ell\ne k}}{n-1}\a\hr{\dfrac{-qz_\ell}{z_k}}.
\eeq
Now, we use iterative application of:
\begin{lem} \label{string}
The projection $P\br{s(z_1)\dots s(z_p)f(z_{p+1})\dots f(z_n)}$ can be written as
\begin{multline}
P\br{s(z_1)\dots s(z_p)f(z_{p+1})\dots f(z_n)} = \\
= P\br{s(z_1)\dots s(z_p)f(z_{p+1})\dots f(z_{n-1})}\Fc(z_1,\dots,z_{n-1};z_n) + {} \\
{} + \suml{k=p+1}{n-1}\tau_{k,p}(z_1,\dots,z_{n-1};z_n)
P\br{s(z_1)\dots s(z_{p})s(z_k)f(z_{p+1})\dots f(z_{k-1})f(z_{k+1})\dots f(z_{n-1})},
\end{multline}
where $\Fc(z_1,\dots,z_{n-1};z_n)$ is defined by \rf{Fc}.
\end{lem}
we reduce evaluation of $P(f(z_1)\dots f(z_n))$ to the problem of calculating $P(s(z_1)\dots s(z_m))$ or $P(s(z_1)\dots s(z_m)f(z_{m+1}))$. Applying Lemma \rfs{string} to $P(s(z_1)\dots s(z_m)f(z_{m+1}))$ once more we reduce the case to the evaluation of $P(s(z_1)\dots s(z_m))$. The latter is achieved by using:
\begin{lem} \label{stringS}
The projection $P\br{s(z_1)\dots s(z_n)}$ admits a decomposition
\beq
P\br{s(z_1)\dots s(z_n)} = P\br{s(z_1)\dots s(z_{n-1})}\Sc(z_1,\dots,z_{n-1};z_n),
\eeq
where $\Sc(z_1,\dots,z_{n-1};z_n)$ is defined by \rf{Sc}.
\end{lem}
One can see that on every step of the iterative application of Lemma \rfs{string}, we either factor projection $P\br{s(z_1)\dots s(z_p)f(z_{p+1})\dots f(z_n)}$ into the product $P\br{s(z_1)\dots s(z_p)f(z_{p+1})\dots f(z_{n-1})}\times$ $\times\Fc(z_1,\dots,z_{n-1};z_n)$ or replace product $f(z_i)f(z_n)$ with the current $s(z_i)$. Thus, sets $I$ and $J$ in the formula \rf{ans} indicate that on the $(n+1-j_1)$-th step, we replaced $f(z_{i_1})f(z_{j_1})$ with $s(z_{i_1})$, on the $(n+1-j_2)$-th step we replaced $f(z_{i_2})f(z_{j_2})$ with $s(z_{i_2})$ etc.
\hfill{$\square$}

\subsection{Proof of Lemma \rfs{string}}
For any current $a(z)=\suml{n\in\Z}{}a_n z^{-n}\,$ let $a^\pm(z)$ denote currents
\begin{align*}
a^+(z) =& \oint\dfrac{a(w)}{1-w/z}\dfrac{dw}{z} = \suml{n>0}{}a_n z^{-n}, \\
a^-(z) =& -\oint\dfrac{a(w)}{1-z/w}\dfrac{dw}{w} = -\suml{n\le0}{}a_n z^{-n}.
\end{align*}
One can see, that
$$
P(f(z)) = f^+(z), \qquad P^-(f(z)) = -f^-(z).
$$
Let us introduce notation that is used in the proof below. Let $\Mmat_c$ be a square matrix of order $n-1$, and $\rF$, $\rS_c$ and $\Vvect_c$ be $(n-1)$-dimensional vectors as follows:
\begin{align}
\label{rP}
\rF &=
\begin{pmatrix}
P(f(z_1)), \dots, P(f(z_{n-1})
\end{pmatrix},
\\ \label{rPs}
\rS_c &=
\begin{pmatrix}
P(s(cz_1)), \dots, P(s(cz_{n-1})
\end{pmatrix},
\end{align}
\beq \label{rV}
\Vvect_c =
\begin{pmatrix}
\dfrac1{1-c^{-1}z_n/z_1},\dots,\dfrac1{1-c^{-1}z_n/z_{n-1}}
\end{pmatrix},
\eeq
\beq \label{rM}
\Mmat_c =
\begin{pmatrix}
\dfrac1{1-c^{-1}z_1/z_1} && \dots && \dfrac1{1-c^{-1}z_1/z_{n-1}} \\
\vdots && \ddots && \vdots \\
\dfrac1{1-c^{-1}z_{n-1}/z_1} && \dots && \dfrac1{1-c^{-1}z_{n-1}/z_{n-1}}
\end{pmatrix},
\eeq
with $c=\pm q^k$.
Next, we introduce the following rational functions of $x$:
\begin{align}
\label{b}
\b(x) &= \dfrac{(1-x)(q^3+x)}{(1-q^2x)(q+x)},
\\ \label{ga}
\ga(x) &= \dfrac{(q^2-x)(q^3+x)(1+qx)}{(1-q^2x)(1+q^3x)(q+x)},
\end{align}
Let also
\beq \label{res-abg}
\a_c(w/z) = \Res\limits_{u=cz}\dfrac{\a(z/u)}{u-w}, \qquad
\b_c(w/z) = \Res\limits_{u=cz}\dfrac{\b(z/u)}{u-w}, \qquad
\ga_c(w/z) = \Res\limits_{u=cz}\dfrac{\ga(z/u)}{u-w},
\eeq
for $c=\pm q^k$. One can see that $\a_c(x) = A(c-x)^{-1}$, $\b_c(x) = B(c-x)^{-1}$, $\ga_c(x) = C(c-x)^{-1}$,
for some constants $A$, $B$, $C$.

Now, by the definition of projection $P$ we have $f(z)=f^+(z) - f^-(z)$.
Using projection properties we obtain
\begin{multline*}
P\br{s(z_1)\dots s(z_p)f(z_{p+1})\dots f(z_n)} = \\
= P\br{s(z_1)\dots s(z_p)f(z_{p+1})\dots f(z_{n-1})}P\br{f(z_n)} - {} \\
{} - P\br{s(z_1)\dots s(z_p)f(z_{p+1})\dots f(z_{n-1})f^-(z_n)}.
\end{multline*}
To calculate the second summand on the right hand side we need to move current $f^-(z_n)$ in front of the product $s(z_1)\dots s(z_p)f(z_{p+1})\dots f(z_{n-1})$. This can be done by inductive application of:
\begin{prop} \label{cP}
The following commutation relations hold:
\begin{itemize}
\item[a)]
let $\a(x)$ and $\a_c(w/z)$ be as in \rf{a} and \rf{res-abg}; then
\begin{multline}
f(z)f^-(w)=\a(z/w)f^-(w)f(z) + {} \\
{} + \a_{q^2}(w/z)f^+(q^2z)f(z) + \a_{-q^{-1}}(w/z)f^+(-q^{-1}z)f(z);
\end{multline}
\item[b)]
let $\b(x)$ and $\b_c(w/z)$ be as in \rf{b} and \rf{res-abg}; then
\begin{multline}
s(z)f^-(w)=\b(z/w)f^-(w)s(z) + {} \\
{} + \b_{q^2}(w/z)f^-(q^2z)s(z) + \b_{-q^{-1}}(w/z)f^+(-q^{-1}z)s(z).
\end{multline}
\end{itemize}
\end{prop}
Informally, the proposition above states that on pushing $f^-(z_n)$ to the left in the product $s(z_1)\dots s(z_p)f(z_{p+1})\dots f(z_{n-1})f^-(z_n)$ some rational functions of $z_n$ arise, which after decomposition into partial fractions appear to have only simple poles on the hyperplanes $z_n=q^2z_i$ and $z_n=-q^{-1}z_i$. Thus the following formula holds:
\begin{multline} \label{Psf1}
P\br{s(z_1)\dots s(z_p)f(z_{p+1})\dots f(z_n)} = \\
= P\br{s(z_1)\dots s(z_p)f(z_{p+1})\dots f(z_{n-1})}P(f(z_n)) - \Vvect_{q^2}\Xvect^\T - \Vvect_{-q^{-1}}\Yvect^\T,
\end{multline}
where $\Xvect$ and $\Yvect$ are $(n-1)$-dimensional vectors with coordinates
\begin{align*}
X_i &= q^{-2}z_i^{-1}\Res\limits_{z_n=q^2z_i}{P\br{s(z_1)\dots s(z_p)f(z_{p+1})\dots f(z_n)}},\\
Y_i &= -qz_i^{-1}\Res\limits_{z_n=-q^{-1}z_i}{P\br{s(z_1)\dots s(z_p)f(z_{p+1})\dots f(z_n)}},
\end{align*}
and $\Vvect_c$ is given by \rf{rV}. Notice that $X_i$ and $Y_i$ do not depend on $z_n$.
Since we have $n-1$ vanishing conditions, $s(z_1)\cdots s(z_p)f(z_{p+1})\cdots f(z_n)$ has simple zeros on the hyperplanes $z_n=z_j$ for $j=1,\dots,n-1$; we can compose a system of linear equations
$$
\Mmat_{q^2}\Xvect^\T+\Mmat_{-q^{-1}}\Yvect^\T=P\br{s(z_1)\dots s(z_p)f(z_{p+1})\dots f(z_{n-1})}\rF^\T,
$$
with $\rF$ and $\Mmat_q$ defined by \rf{rP}, \rf{rM}.
Expressing $\Xvect$ from the system above and substituting it into equation \rf{Psf1}, we get
\begin{multline*}
P\br{s(z_1)\dots s(z_p)f(z_{p+1})\dots f(z_n)} = \\
= P\br{s(z_1)\dots s(z_p)f(z_{p+1})\dots f(z_{n-1})}\hr{P(f(z_n)) - \Vvect_{q^2}\Mmat_{q^2}^{-1}\rF^\T} - {}\\
{} - \hr{\Vvect_{-q^{-1}}-\Vvect_{q^2}\Mmat_{q^2}^{-1}\Mmat_{-q^{-1}}}\Yvect^\T.
\end{multline*}

\begin{prop} \label{VM}
We have the following equalities:
\begin{itemize}
\item[a)]
let $\Fc(z_1,\dots,z_{n-1};z_n)$ be as in \rf{Fc}; then
\beq
P(f(z_n)) - \Vvect_{q^2}\Mmat_{q^2}^{-1}\rF^\T = \Fc(z_1,\dots,z_{n-1};z_n);
\eeq
\item[b)]
let $\la_i(z_1,\dots,z_{n-1};z_n)$ be as in \rf{la}; then
\beq
\hr{\Vvect_{-q^{-1}} - \Vvect_{q^2}\Mmat_{q^2}^{-1}\Mmat_{-q^{-1}}}\Yvect^\T = \suml{i=1}{n-1}\la_i(z_1,\dots,z_{n-1};z_n)Y_i.
\eeq
\end{itemize}
\end{prop}
Applying Proposition \rfs{VM} we arrive at the equation
\begin{multline*}
P\br{s(z_1)\dots s(z_p)f(z_{p+1})\dots f(z_n)} = \\
= P\br{s(z_1)\dots s(z_p)f(z_{p+1})\dots f(z_{n-1})}\Fc(z_1,\dots,z_{n-1};z_n) - \suml{i=1}{n-1}\la_i(z_1,\dots,z_{n-1};z_n)Y_i.
\end{multline*}
Now let us notice that $Y_i=0$ for $i=1,\dots,p$, since the product $s(z_1)\dots s(z_p)f(z_{p+1})\dots f(z_n)$ does not have poles on the hyperplanes $z_n=-q^{-1}z_i$ for $i=1,\dots,p$. Taking into account that the action of projection $P(f(z)f(w))$ is well defined on every highest weight $\advadva$-module for any arbitrary $z$ and $w$, we obtain the following commutation relations:
\begin{align}
P(f(z)f(w)) &= \a\hr{z/w}P(f(w)f(z)),  \label{Palpha}  \\
P(s(z)f(w)) &= \b\hr{z/w}P(f(w)s(z)),  \label{Pbeta}   \\
P(s(z)s(w)) &= \ga\hr{z/w}P(s(w)s(z)), \label{Pgamma}
\end{align}
with $\a(x)$, $\b(x)$, $\ga(x)$ defined by \rf{a}, \rf{b} and \rf{ga} respectively. Notice, that in the relations above one should take an analytic continuation of the left hand or the right hand side when necessary. Using relation \rf{Palpha}, we rewrite $Y_k$ as follows:
\begin{multline*}
Y_k = \prodl{i=p+1}{n-1}\a\hr{-qz_i/z_k} \times {} \\
\\ {} \times (-qz_i^{-1})\Res\limits_{z_n=-q^{-1}z_i}
P\br{s(z_1) \dots s(z_p)f(z_k)f(z_{p+1}) \dots f(z_{k-1})f(z_{k+1}) \dots f(z_n)},
\end{multline*}
where $p<k<n$, and $\a(x)$ is defined by \rf{a}.
Finally, using relation \rf{Pbeta} and mutual commutativity of the projection
$P$ and of the operator of taking a residue, we obtain
\begin{multline*}
Y_k = \prodl{i=p+1}{k-1}\a\hr{z_i/z_k} \prodl{\substack{i=p+1 \\ i\ne k}}{n-1}\a\hr{-qz_i/z_k} \times {} \\
\\ {} \times P\br{s(z_1) \dots s(z_p)s(z_k)f(z_{p+1}) \dots f(z_{k-1})f(z_{k+1}) \dots f(z_{n-1})}.
\end{multline*}
Hence the statement of the lemma holds.
\hfill{$\square$}

\subsection{Proof of Lemma \rfs{stringS}}

We start with:
\begin{prop} \label{decomp}
Current $s(z)$ admits the following expansion:
\beq \label{sdec}
s(z)=P^+(s(z))+\hr{\dfrac{\hr{q-q^{-1}}}{\hr{q-1+q^{-1}}}f^-\hr{-q^{-1}z}f(z)}^+-s^-(z).
\eeq
\end{prop}
Hence the projection $P\br{s(z_1)\dots s(z_n)}$ can be evaluated via pushing $f^-\hr{-q^{-1}z_n}$ and $s^-(z_n)$ in front of the product $s(z_1)\dots s(z_n)$. We use Proposition \rfs{cP} and:
\begin{prop} \label{cP2}
The following commutation relation holds:
\begin{multline}
s(z)s^-(w)=\ga(z/w)s^-(w)s(z) + {} \\
{} + \ga_{q^2}(w/z)s^+(q^2z)s(z) + \ga_{-q^3}(w/z)s^+(-q^3z)s(z) + \ga_{-q^{-1}}(w/z)s^-(-q^{-1}z)s(z),
\end{multline}
where $\ga(x)$ and $\ga_c(w/z)$ are given by \rf{ga} and \rf{res-abg}.
\end{prop}
Then we arrive at the following equation:
\be \label{sss}
P\br{s(z_1)\dots s(z_n)} = P\br{s(z_1)\dots s(z_{n-1})}P(s(z_n)) - \Vvect_{q^2}\Xvect^\T - \Vvect_{-q^{-1}}\Yvect^\T - \Vvect_{-q^3}\Zvect^\T,
\ee
where $\Xvect$, $\Yvect$ and $\Zvect$ are $(n-1)$-dimensional vectors with coordinates
\begin{align*}
X_i &= q^{-2}z_i^{-1}\Res\limits_{z_n=q^2z_i}{P\br{s(z_1)\dots s(z_n)}}, \\
Y_i &= -qz_i^{-1}\Res\limits_{z_n=-q^{-1}z_i}{P\br{s(z_1)\dots s(z_n)}}, \\
Z_i &= -q^{-3}z_i^{-1}\Res\limits_{z_n=-q^3z_i}{P\br{s(z_1)\dots s(z_n)}},
\end{align*}
and $\Vvect_q$ is defined by \rf{rV}.
Notice that $Y_i = 0$ for $i=1,\dots,n-1$ since the product $s(z_1)\dots s(z_n)$ does not have poles on the hyperplanes $z_n=-q^{-1}z_i$, while $X_i$ and $Z_i$ are functions independent of $z_n$. Regarding $2n-2$ vanishing conditions, namely $s(z_1)\dots s(z_n)=0$ on the hyperplanes $z_n=z_i$ and $z_n=-qz_i$, we can compose a system of $2n-2$ linear equations with $2n-2$ variables $X_i$ and $Z_i$ for $i=1,\dots,n-1$.
Let
$$
\Mmat:=
\begin{pmatrix}
\Mmat_{q^2} && \Mmat_{-q^3} \\
\Mmat_{-q} && \Mmat_{q^2}
\end{pmatrix},
\qquad
\Vvect:=
\left(
\Vvect_{q^2},\;\;\Vvect_{-q^3}
\right),
\qquad \text{and} \qquad
\rS:=
\left(
\rS_1^\T,\;\; \rS_{-q}^\T
\right),
$$
where $\Mmat_c$, $\Vvect_c$ and $\rS_c$ are defined by \rf{rM}, \rf{rV} and \rf{rPs} respectively.
Thus we can present our system of equations as
$$
\Mmat
\begin{pmatrix}
\Xvect^\T \\
\Zvect^\T
\end{pmatrix}
=P\br{s(z_1)\dots s(z_{n-1})}\rS^\T.
$$
Solving the system and substituting $\Xvect$ and $\Zvect$ into the equation \rf{sss} we get
$$
P\br{s(z_1)\dots s(z_n)} = P\br{s(z_1)\dots s(z_{n-1})}\hr{P(s(z_n)) - \Vvect\Mmat^{-1}\rS^\T}.
$$
Now we finish the proof by applying:
\begin{prop} \label{VM2}
We have the following equality:
\beq
P(s(z_n)) - \Vvect\Mmat^{-1}\rS^\T=\Sc(z_1,\dots,z_{n-1};z_n),
\eeq
where $\Sc(z_1,\dots,z_{n-1};z_n)$ is given by \rf{Sc}.
\end{prop}
\hfill{$\square$}

\subsection{Proof of Propositions \rfs{cP} and \rfs{cP2}}

Since the proofs are almost the same we present only the proof of Proposition \rfs{cP} part $a)$.

\noindent
Let us recall relation \rf{ff}:
\pbe
(q^2z-u)(z+qu)f(z)f(u)=(z-q^2u)(qz+u)f(u)f(z).
\pee
Since product $f(z)f(u)$ has simple poles on the planes $u=q^2z$ and $u=-q^{-1}z$,
we derive the following equality of two holomorphic functions
\pbe
f(z)f(u) + \dfrac1{1-q^{-2} u/z}t(z) + \dfrac1{1+q u/z}s(z) = \\ =
\a(z/u)f(u)f(z) - \dfrac{q^2z}{u}\dfrac1{1-q^2 z/u}t(z) + \dfrac{z}{qu}\dfrac1{1+q^{-1} z/u}s(z).
\pee
As a matter of fact, the coefficients of $t(z)$ and $s(z)$ on the left hand side and the right hand side are the same rational functions, but they are represented as power series in $u/z$ and $z/u$ respectively for the reasons mentioned in Section \rfs{subsect_compl}. Hence we get
$$
f(z)f(u)=\dfrac{(q^2-z/u)(q^{-1}+z/u)}{(1-q^2z/u)(1+q^{-1}z/u)}f(u)f(z)
- \d(q^{-2}u/z)t(z) - \d(-qu/z)s(z).
$$
Under the action of $\displaystyle-\oint\dfrac1{1-w/u}\dfrac{du}{u}$ the previous equality turns into
$$
\displaystyle
f(z)f^-(w)=-\oint \a(z/u)\dfrac1{1-w/u}f(u)f(z)\dfrac{du}{u}
+ \dfrac1{1-q^{-2}w/z}t(z) + \dfrac1{1+qw/z}s(z).
$$
Since the integral above can be rewritten as
$$
\a(z/w)f^-(w)f(z) + \a(w/z)f^-(q^2z)f(z) + \b(w/z)f^-(-q^{-1}z)f(z),
$$
using relations \rf{s}, \rf{t}, we derive the statement of the proposition.

\hfill{$\square$}

\subsection{Proof of Propositions \rfs{VM} and \rfs{VM2}}

To prove part $a)$ of the proposition we only need to find out that $\Vvect_q\Mmat_q^{-1}$ is an $(n-1)$-dimensional vector $\rW_q$ with
$$
W_q^{(k)} = \prodl{\substack{i=1 \\ i\ne k}}{n-1}\dfrac{z_n-z_i}{z_k-z_i}\prodl{i=1}{n-1}\dfrac{z_k-qz_i}{z_n-qz_i}
$$
for $k=1,\dots,n-1$. Regarding $W_q^{(k)}$ as a rational function over $z_n$, we observe that it has simple poles only on the planes $z_n=qz_i$ for $i=1,\dots,n-1$. We also know that $W_q^{(k)}$ is a rational function of degree $-1$ in $z_n$. Let us notice that $\Vvect_q$ is equal to the $i$-th row of matrix $\Mmat_q$ if $z_n=z_i$. Therefore $W_q^{(k)}$ equals $1$ on the hyperplane $z_n=z_k$ and equals $0$ on the hyperplanes $z_n=z_i$ for $i=1,\dots,n-1,\;i\ne k$. These conditions totally determine $W_q^{(k)}$. Now, replacing $q$ by $q^2$ we derive statement $a)$ of the proposition.

To prove the statement of part $b)$ we need to show that
\beq \label{la1}
\la_k(z_n;z_1,\dots,z_{n-1})=\dfrac1{1+qz_n/z_k}-\Vvect_{q^2}\Mmat_{q^2}^{-1}\Mmat_{-q^{-1}}^{(k)},
\eeq
where $\Mmat_{-q^{-1}}^{(k)}$ stands for the $k$-th column of matrix $\Mmat_{-q^{-1}}$. Regarded as a rational function of $z_n$, $\la_i(z_1,\dots,z_{n-1};z_n)$ has simple poles only on the hyperplanes $z_n=q^2z_i$ for $i=1,\dots,n-1$ and $z_n=-q^{-1}z_k$. Since vector $\Vvect_q$ is equal to the $i$-th row of the matrix $\Mmat_q$ for $z_n=z_i$, $\la_k(z_1,\dots,z_{n-1};z_n)$ has simple zeros on the hyperplanes $z_n=z_i, \; i=1,\dots,n-1$. Considering the fact that $\la_k(z_1,\dots,z_{n-1};z_n)$ is a rational function of degree $-1$ in $z_n,$ we can present it in the following form:
\pbe
\la_i(z_1,\dots,z_{n-1};z_n) =\dfrac1{qz_n+z_i}\prodl{j=1}{n-1}
\hr{\dfrac{z_n-z_j}{z_n-q^2z_j}}\cdot C(z_1,\dots,z_{n-1}).
\pee
Equality \rf{la1} implies that
\pbe
\lim\limits_{z_n\to -q^{-1}z_k}\bigl((1+qz_n/z_k)\la_k(z_1,\dots,z_{n-1};z_n)\bigr)=1.
\pee
We finish the proof by obtaining $C(z_1,\dots,z_{n-1})$ from the latter condition.

The proof of Proposition \rfs{VM2} is an easy modification of the proof of Proposition \rfs{VM} part $a)$.

\hfill{$\square$}

\subsection{Proof of Proposition \rfs{decomp}}
The statement of the proposition follows from equality \rf{sdecomp}, projection properties \rf{Pdef} and the statement \emph{b)} of Theorem \rfs{Th1}.
\hfill{$\square$}

\section{Examples} \label{Examples}
In this section we work out particular cases of the formula \rf{ans} for $n=2,3,4$. Recall that all rational functions are considered as formal power series converging in the region $|z_1|\gg\dots\gg|z_n|$.
First of all let us recall that
$$
P(f(z)) = f^+(z) = \suml{n>0}{}f_nz^{-n},
$$
$$
P(s(z)) = -\dfrac{1}{q+q^{-2}}\suml{n>0}{}\hr{qf_nf_0 - f_0f_n + f_1f_{n-1} - q^{-1}f_{n-1}f_1}z^{-n}.
$$
Now, for $n=2$ we have
$$
P(f(z_1)f(z_2))= \Fc^1_{\hc{\es},\hc{\es}}(z_1,z_2)\Fc^2_{\hc{\es},\hc{\es}}(z_1,z_2)
+ \tau^1_{\hc{1},\hc{2}}(z_1,z_2)\Sc(z_1).
$$
Since $\Sc(z)=P(s(z))$ and $\Fc^k_{\hc{\es},\hc{\es}}(z_1,\dots,z_n) = \Fc(z_1,\dots,z_{n-1};z_n)$, we get
\begin{align*}
\Sc(z_1) &= P(s(z_1)), \\
\Fc^1_{\hc{\es},\hc{\es}}(z_1,z_2) &= f^+(z_1), \\
\Fc^2_{\hc{\es},\hc{\es}}(z_1,z_2) &= f^+(z_2) - \dfrac{q^2-1}{q^2-z_2/z_1}f^+(z_1).
\end{align*}
Here
$$
\tau^1_{\hc{1},\hc{2}}(z_1,z_2) = -\dfrac{(1+q^3)(1-z_2/z_1)}{(1+q)(q^2-z_2/z_1)(1+qz_2/z_1)}.
$$

For $n=3$, formula \rf{ans} turns into
\begin{align*}
P(f(z_1)f(z_2)f(z_3))
   &= \Fc^1_{\es,\es}(z_1,z_2,z_3)\Fc^2_{\es,\es}(z_1,z_2,z_3)\Fc^3_{\es,\es}(z_1,z_2,z_3) + {} \\
{} &+ \tau^1_{\hc{1},\hc{2}}(z_1,z_2,z_3)\Sc(z_1)\Fc^3_{\hc{1},\hc{2}}(z_1,z_2,z_3) + {} \\
{} &+ \tau^1_{\hc{1},\hc{3}}(z_1,z_2,z_3)\Sc(z_1)\Fc^2_{\hc{1},\hc{3}}(z_1,z_2,z_3) + {} \\
{} &+ \tau^1_{\hc{2},\hc{3}}(z_1,z_2,z_3)\Sc(z_2)\Fc^1_{\hc{2},\hc{3}}(z_1,z_2,z_3).
\end{align*}
Here again
\pbeq
\Sc(z_i) = P(s(z_i)),
\peeq
and $\Fc^k_{\hc{\es},\hc{\es}}(z_1,\dots,z_n) = \Fc(z_1,\dots,z_{n-1};z_n)$. Therefore we obtain
\begin{align*}
\Fc^1_{\es,\es}(z_1,z_2,z_3) &= f^+(z_1), \\
\Fc^2_{\es,\es}(z_1,z_2,z_3) &= f^+(z_2) - \dfrac{q^2-1}{q^2-z_2/z_1}f^+(z_1), \\
\Fc^3_{\es,\es}(z_1,z_2,z_3) &= f^+(z_3) -
\dfrac{(1-z_3/z_2)(q^2-1)(1-q^2z_2/z_1)}{(1-z_2/z_1)(q^2-z_3/z_1)(q^2-z_3/z_2)}f^+(z_1) - {} \\
&\phantom{{} = f^+(z_3) } - \dfrac{(1-z_3/z_1)(q^2-z_2/z_1)(q^2-1)}{(1-z_2/z_1)(q^2-z_3/z_1)(q^2-z_3/z_2)}f^+(z_2).
\end{align*}
One can also see that
\begin{align*}
\Fc^3_{\hc{1},\hc{2}}(z_1,z_2,z_3) &= \Fc^3_{\es,\es}(z_1,z_2,z_3), \\
\Fc^2_{\hc{1},\hc{3}}(z_1,z_2,z_3) &= \Fc^2_{\es,\es}(z_1,z_2,z_3), \\
\end{align*}
and
$$
\Fc^1_{\hc{2},\hc{3}}(z_1,z_2,z_3) = f^+(z_1) - \dfrac{z_2}{z_1}\dfrac{1-q^2}{1-q^2z_2/z_1}f^+(z_2).
$$
The following expressions are obtained from corresponding rational functions by expansion in power series in the domain $|z_1|\gg|z_2|\gg|z_3|$:
\begin{align*}
\tau^1_{\hc{1},\hc{2}}(z_1,z_2,z_3) &= -\dfrac{(1+q^3)(1-z_2/z_1)}{(1+q)(q^2-z_2/z_1)(1+qz_2/z_1)}, \\
\tau^1_{\hc{1},\hc{3}}(z_1,z_2,z_3) &=
-\dfrac{(1+q^3)(1-z_3/z_1)(1-z_3/z_2)}{(1+q)(1+qz_3/z_1)(q^2-z_3/z_1)(q^2-z_3/z_2)} \times {} \\
&\phantom{={}} {} \times \dfrac{(q^2-z_2/z_1)(1+q^3z_2/z_1)}{(1-z_2/z_1)(q^3+z_2/z_1)}, \\
\tau^1_{\hc{2},\hc{3}}(z_1,z_2,z_3) &=
-\dfrac{(1+q^3)(1-z_3/z_1)(1-z_3/z_2)}{(1+q)(1+qz_3/z_1)(q^2-z_3/z_1)(q^2-z_3/z_2)} \times {} \\
&\phantom{={}} {} \times \dfrac{(1-q^2z_2/z_1)(q+z_2/z_1)(1+q^3z_2/z_1)}{(1-z_2/z_1)(1+qz_2/z_1)(q^3+z_2/z_1)}.
\end{align*}

Finally, we obtain formula \rf{ans} in the case $n=4$:
\begin{align*}
P(f(z_1)f(z_2)&f(z_3)f(z_4)) = \\
   &= \Fc^1_{\es,\es}(z_1,z_2,z_3,z_4)\Fc^2_{\es,\es}(z_1,z_2,z_3,z_4)
      \Fc^3_{\es,\es}(z_1,z_2,z_3,z_4)\Fc^4_{\es,\es}(z_1,z_2,z_3,z_4) + {} \\
{} &+ \suml{a,b,i,j}{}\tau^1_{\hc{i},\hc{j}}(z_1,z_2,z_3,z_4)\Sc(z_i)
      \Fc^a_{\hc{i},\hc{j}}(z_1,z_2,z_3,z_4)\Fc^b_{\hc{i},\hc{j}}(z_1,z_2,z_3,z_4) + \\
{} &+ \tau^1_{\hc{1,2},\hc{4,3}}(z_1,z_2,z_3,z_4)\tau^2_{\hc{1,2},\hc{4,3}}(z_1,z_2,z_3,z_4)
      \Sc(z_1)\Sc(z_1;z_2) + {} \\
{} &+ \tau^1_{\hc{2,1},\hc{4,3}}(z_1,z_2,z_3,z_4)\tau^2_{\hc{2,1},\hc{4,3}}(z_1,z_2,z_3,z_4)
      \Sc(z_2)\Sc(z_2;z_1) + {} \\
{} &+ \tau^1_{\hc{3,1},\hc{4,2}}(z_1,z_2,z_3,z_4)\tau^2_{\hc{3,1},\hc{4,2}}(z_1,z_2,z_3,z_4)
      \Sc(z_3)\Sc(z_3;z_1),
\end{align*}
where the sum is taken over all permutations $\hc{a,b,i,j}$ of an ordered set $\hc{1,2,3,4}$ such that $i<j$ and $a<b$.
In the latter example
\begin{align*}
\Sc(z) &= P(s(z)), \\
\Sc(z_1;z_2) &= P(s(z_2)) - \dfrac{(q-1)(q^3+1)(q+z_2/z_1)}{(q^2-z_2/z_1)(q^3+z_2/z_1)}P(s(z_1)) - {}\\
             &\phantom{{} = P(s(z_2))} - \dfrac{q^2(q^2-1)(1-z_2/z_1)}{(q^2-z_2/z_1)(q^3+z_2/z_1)}P(s(-qz_1)),\\
\Fc^4_{\hc{i},\hc{j}}(z_1,z_2,z_3,z_4) &= \Fc^4_{\es,\es}(z_1,z_2,z_3,z_4) = \Fc(z_1,z_2,z_3;z_4),\\
\Fc^3_{\hc{i},\hc{j}}(z_1,z_2,z_3,z_4) &= \Fc^3_{\es,\es}(z_1,z_2,z_3,z_4) = \Fc(z_1,z_2;z_3),\\
\Fc^2_{\hc{1},\hc{j}}(z_1,z_2,z_3,z_4) &= \Fc^2_{\es,\es}(z_1,z_2,z_3,z_4) = \Fc(z_1;z_2),\\
\Fc^1_{\hc{2},\hc{j}}(z_1,z_2,z_3,z_4) &= \Fc(z_2;z_1),\\
\Fc^2_{\hc{3},\hc{4}} &= \Fc(z_3,z_1;z_2),\\
\Fc^1_{\hc{3},\hc{4}} &= \Fc(z_3;z_1),
\end{align*}
for all $1\le i<j \le 4$, and
\begin{align*}
\tau^1_{\hc{1,2},\hc{4,3}}(z_1,z_2,z_3,z_4) &= -\la_1(z_1,z_2,z_3;z_4)\a(-qz_2/z_1)\a(-qz_3/z_1), \\ \tau^2_{\hc{1,2},\hc{4,3}}(z_1,z_2,z_3,z_4) &= -\la_2(z_1,z_2;z_3), \\
\tau^1_{\hc{2,1},\hc{4,3}}(z_1,z_2,z_3,z_4) &= -\la_2(z_1,z_2,z_3;z_4)\a(z_1/z_2)\a(-qz_1/z_2)\a(-qz_3/z_2), \\
\tau^2_{\hc{2,1},\hc{4,3}}(z_1,z_2,z_3,z_4) &= -\la_1(z_2,z_1;z_3), \\
\tau^1_{\hc{3,1},\hc{4,2}}(z_1,z_2,z_3,z_4) &= -\la_3(z_1,z_2,z_3;z_4)\a(z_1/z_3)\a(z_2/z_3)\a(-qz_1/z_3)\a(-qz_2/z_3), \\
\tau^2_{\hc{3,1},\hc{4,2}}(z_1,z_2,z_3,z_4) &= -\la_1(z_3,z_1;z_2),
\end{align*}
where $\la_k(z_1,\dots,z_{n-1};z_n)$ and $\a(x)$ are given by \rf{la} and \rf{a} respectively.

\section*{Acknowledgements}
The authors thank Stanislav Pakuliak for fruitful discussions. The first author was supported by interdisciplinary RFBR grant 09-01-12185-ofi-m, joint grant 09-02-90493-Ukr-f, and grant for support of Scientific schools 3036.2008.2-NSh. The second author was supported by RFBR grant 08-01-00667, joint CNRS-RFBR grant 09-01-93106-NCNIL, and by Federal Agency for Science and Innovations of Russian Federation under contract 14.740.11.0081.

\end{document}